\documentclass[10pt]{amsart}
\usepackage{amssymb}
\usepackage{amsfonts}
\usepackage{amsmath}
\usepackage{graphicx}
\usepackage{amsmath,amssymb,amscd,verbatim}
\usepackage{amsfonts}
\usepackage{color}
\usepackage[all]{xy}

\theoremstyle{plain}
\newtheorem{theorem}{Theorem}[section]
\newtheorem{corollary}[theorem]{Corollary}
\newtheorem{lemma}[theorem]{Lemma}
\newtheorem{proposition}[theorem]{Proposition}

\theoremstyle{definition}

\newtheorem{remark}[theorem]{Remark}
\newtheorem{example}[theorem]{Example}

\newtheorem{conjecture}[theorem]{Conjecture}
\newtheorem{remarks/questions}[theorem]{Remarks/Questions}

\begin{document}
\title[Completely integrally closed PVMDs]{Completely integrally closed
Pr\"{u}fer $v$-multiplication domains}
\author{D. D. Anderson}
\author{David F. Anderson}
\author{Muhammad Zafrullah}
\address{Department of Mathematics\\
The University of Iowa\\
Iowa City, IA 52242 U.S.A.}
\email{dan-anderson@uiowa.edu}
\address{Department of Mathematics\\
The University of Tennessee\\
Knoxville, TN 37996-1320 U.S.A.}
\email{anderson@math.utk.edu}
\address{Department of Mathematics\\
Idaho State University\\
Pocatello, ID 83209 U.S.A.}
\email{mzafrullah@usa.net}

1/23/17

\begin{abstract}
We study the effects on $D$ of assuming that the power series ring $D[[X]]$ is a $v$-domain or a PVMD.
We show that a PVMD $D$ is completely integrally closed if
and only if  $\bigcap_{n = 1}^{\infty}(I^{n})_{v} = (0)$ for every proper $t$-invertible \mbox{$t$-ideal} $I$ of $D$.
Using this, we show that if $D$ is an AGCD domain, then $D[[X]]$ is
integrally closed if and only if $D$ is a completely integrally
closed PVMD with torsion \mbox{$t$-class} group. We also determine several classes of PVMDs
for which being Archimedean is equivalent to being completely integrally closed
and give some new characterizations of integral domains related to Krull domains.
\end{abstract}

\maketitle


\section*{Introduction}

The aim of this paper is to show that $D$ is a $v$-domain when $D[[X]]$ is a $v$-domain
and to prove the following two results and record their consequences.
Throughout, $D$ is an integral domain with quotient field $K$.
Other necessary definitions will be provided later.

\begin{theorem} \label{A}
Let $D$ be an integral domain that is an intersection of
localizations at divisorial prime ideals. If $D[[X]]$ is a Pr\"{u}fer $v$-multiplication domain (PVMD),
then $D$ is a $v$-domain that is an
intersection of essential discrete rank-one valuation domains, and thus is completely integrally closed.
\end{theorem}

\begin{theorem}  \label{B}
A PVMD $D$ is completely integrally closed if and only if
$\bigcap_{n = 1}^{\infty}(I^{n})_{v}=(0)$ for every proper $t$-invertible $t$-ideal $I$ of $D$.
\end{theorem}

Several classes of integral domains of interest, such as Noetherian, Krull,
Mori, and the so-called H-domains that have maximal $t$-ideals divisorial,
fall under the umbrella of integral domains that are intersections of localizations
at divisorial prime ideals. We show that an H-domain $D$ is a Krull domain
if and only if $D[[X]]$ is a PVMD. As a consequence of Theorem \ref{B}, we
show that if $D$ is an almost GCD (AGCD) domain, then $D[[X]]$ is integrally closed
if and only if $D$ is a completely integrally closed PVMD with torsion $t$-class group.
We also show that if $D$ is an AGCD domain such that $D[[X]]$
is integrally closed and every nonzero nonunit of $D$ has only
finitely many minimal prime ideals, then $D$ is a locally finite intersection
of rank-one valuation domains. We isolate a property of integral domains of finite
$t$-character and use it in combination with complete integral closure to give some
new characterizations of integral domains related to Krull domains and their
generalizations. We also answer a recently asked question about
the ring of power series over a Krull-like PVMD.

As our work involves star operations,
it seems pertinent to give the reader an idea of some of the notions
involved. Let $D$ be an integral domain with quotient field $K$, and let
$F(D)$ (resp., $f(D)$) be the set of nonzero fractional ideals (resp.,
nonzero finitely generated fractional ideals) of $D$.

A \emph{star operation} $\ast $ on $D$ is a function $\ast \colon
F(D)\longrightarrow F(D)$ that satisfies the following properties for every $I, J \in F(D)$ and $0 \neq
x \in K$:

(i) $(x)^{\ast} = (x)$ and $(xI)^{\ast} = xI^{\ast}$,

(ii) $I \subseteq I^{\ast}$, and $I^{\ast} \subseteq J^{\ast}$ whenever
$I \subseteq J$, and

(iii) $(I^{\ast})^{\ast} = I^{\ast}$.

\vspace{0.1cm} \noindent An $I \in F(D)$ is called a \emph{$
\ast $-ideal} if $I^{\ast} = I$ and a \emph{$\ast$-ideal of finite type} if
$I = J^{\ast}$ for some $J \in f(D)$. A star operation $\ast $ is said to be of
\emph{finite character} if $I^{\ast} = \bigcup \{ \, J^{\ast} \mid J\subseteq I$
and $J \in f(D) \, \}$. For $I \in F(D)$, let $I_d = I$, $I^{-1} = (D :_K I) = \{ \, x\in K\mid xI\subseteq D \, \}$,
$I_{v}=(I^{-1})^{-1}$, $I_{t}=\bigcup \{ \, J_{v} \mid J \subseteq I$ and $J \in
f(D) \, \}$, and $I_{w} =  \{ \, x \in K \mid xJ \subseteq I$ for some $J \in f(D)$ with $
J_{v} = \, D  \, \}$. A $v$-ideal is sometimes also called a \emph{divisorial ideal}.
The functions defined by $I \mapsto I_{d}$, $I \mapsto I_{v}$, $I \mapsto I_{t}$, and
$I \mapsto I_{w}$ are all examples of star operations. Given two star
operations $\ast _{1},\ast _{2}$ on $D$, we say that $\ast _{1} \leq \ast _{2}$ if
$I^{\ast_{1}} \subseteq I^{\ast_{2}}$ for every $I \in F(D)$. Note that $\ast
_{1} \leq \ast _{2}$ if and only if $(I^{\ast_{1}})^{\ast_{2}} = (I^{\ast_{2}})^{\ast_{1}} =
I^{\ast_{2}}$ for every $I \in F(D)$.  The $d$-operation, $t$-operation, and $w$-operation
all have finite character, $d \leq \rho \leq v$ for every star operation $\rho$,
and $\rho \leq t$ for every star operation $\rho$ of finite character.
We will often use the two facts that $(IJ)^{\ast} = (IJ^{\ast})^{\ast} = (I^{\ast}J^{\ast})^{\ast}$
for every star operation $\ast$ and $I, J \in F(D)$ and $I_v = I_t$ for every $I \in f(D)$.
An $I \in F(D)$ is said to be \emph{$\ast $-invertible} if $(II^{-1})^{\ast} = D$.
If $I$ is $\ast $-invertible for $\ast $ of finite character, then both $I^{\ast}$
and $I^{-1}$ are $v$-ideals of finite type.
The reader in need of more introduction may consult \cite{z00} or \cite[Sections 32 and 34]
{gilmer}.

For a star operation $\ast$, a \emph{maximal $\ast$-ideal} is an integral $\ast$-ideal that is
maximal among proper integral $\ast$-ideals. Let $\ast $-Max$(D)$ be
the set of maximal $\ast $-ideals of $D$. For a star operation
$\ast$ of finite character, it is well known that a maximal $\ast$-ideal is
a prime ideal; every proper integral $\ast $-ideal is contained in a maximal $\ast$-ideal;
and $\ast$-Max$(D)\neq \emptyset $ if $D$ is not a field. Moreover, $t$-Max$
(D) =$ \mbox{$w$-Max$(D)$;} $I_{w} = \bigcap _{M \in t\text{-Max}(D)}ID_{M}$ for every $I \in
F(D)$; and $I_{w}D_{M} = ID_{M}$ for every $I \in F(D)$ and $M \in t$-Max$(D)$.

Recall that an integral domain $D$ with quotient field $K$ is
\emph{completely integrally closed} if whenever $rx^n \in D$ for
$x \in K$, $0 \neq r \in D$, and every integer $n \geq 1$, then $x \in D$.
Equivalently, $D$ is completely integrally closed if and only if $(II^{-1})_v = D$ for every $I \in F(D)$ \cite[Theorem 34.3]{gilmer}.
We will use the well-known facts that a completely integrally closed domain is integrally closed \cite[Theorem 13.1(2)]{gilmer},
an intersection of completely integrally closed domains is completely integrally closed,
$D[[X]]$ is completely integrally closed if and only if $D$ is completely integrally closed \cite[Theorem 13.9]{gilmer},
a Krull domain is completely integrally closed, and a valuation domain $D$ is completely integrally closed
if and only if $D$ has rank at most one \cite[Theorem 17.5(3)]{gilmer}.

We say that an integral domain $D$ is a \emph{Pr\"{u}fer $v$-multiplication domain} (\emph{PVMD}) if every nonzero finitely
generated ideal $I$ of $D$ is $t$-invertible, i.e., $(II^{-1})_{t} = D$ for every $I \in f(D)$.
If $D_P$ is a valuation domain for a nonzero prime ideal $P$ of $D$, then $P$ is necessarily a $t$-ideal of $D$.
For PVMDs, the converse is true. Indeed, Griffin \cite[Theorem 5]{Gr} showed that $D$ is a PVMD if and only if $D_M$ is a
valuation domain for every maximal $t$-ideal $M$ of $D$. As indicated in \cite{z00},
Kang \cite{Kan} showed that an
integrally closed domain $D$ is a PVMD if and only if $t = w$ over $D$. An integral
domain $D$ is a \emph{$v$-domain} if every nonzero finitely
generated ideal $I$ of $D$ is \mbox{$v$-invertible}, i.e., $(II^{-1})_{v} = D$ for every $I \in f(D)$.
Equivalently, $D$ is a PVMD (resp., $v$-domain) if and only if for every $I \in f(D)$,
there is a $J \in f(D)$ (resp., \mbox{$J \in F(D)$}) such that $(IJ)_v = D$. Thus, a PVMD is a $v$-domain.
Note that a $v$-domain (and hence a PVMD) is integrally closed and a completely integrally closed domain is a \mbox{$v$-domain}.
It can be shown that $D$ is a PVMD (resp., \mbox{$v$-domain}) if
and only if every nonzero two-generated ideal of $D$ is \mbox{$t$-(resp., $v$-)invertible}
 \mbox{\cite[Theorem 2.2]{AAFZ}}. From this, it is easy to conclude that a
\mbox{$v$-domain} $D$ is a PVMD if and only if $aD \cap bD$ is a $v$-ideal of finite type
for every $0 \neq a, b\in D$.

A valuation overring $V$ of an integral domain $D$ is called an
\emph{essential valuation domain} if $V = D_P$ for some prime ideal $P$ of $D$ ($P$ is called an \emph{essential} or \emph{valued-prime ideal}).
We call $D$ an \emph{essential domain} if $D = \bigcap_{P \in F}D_P$ for some family $F$ of essential prime ideals of $D$.
An essential domain is integrally closed, and a PVMD is an essential domain since $D = \bigcap_{M \in t\text{-Max}
(D)}D_M$.

An integral domain $D$ is called a \emph{GCD domain} if $(a) \cap (b)$ is principal for every
$ 0 \neq a,b\in D$ and an \emph{almost GCD} (\emph{AGCD}) \emph{domain} if for every
$ 0 \neq a,b\in D$, there is an integer $n \geq 1$ such that $(a^n) \cap (b^n)$ is principal.
Thus, a GCD domain is a a PVMD in which every $v$-ideal of finite type is principal, and a GCD domain is an AGCD domain.
AGCD domains were introduced in
\cite{Z} and further studied in \cite{AZ}. It is well known that $D$ is an
AGCD domain if and only if for every
$0 \neq a_{1}, \ldots ,a_{s} \in D$, there is an integer $k \geq 1$ such that
$(a_{1}^{k}, \ldots, a_{s}^{k})_{v}$ is principal \cite[Remark after Lemma 3.3]{AZ}.

The set $t$-inv$(D)$ of $t$-invertible fractional $t$-ideals of $D$ is an abelian group under the \mbox{$t$-multiplication}
$I \ast J = (IJ)_t$. Its subset $P(D)$ of nonzero principal fractional ideals is
a subgroup of $t$-inv$(D)$. The quotient group $t$-inv$(D)/P(D)$ is
called the \emph{class group} (or \emph{$t$-class group}) of $D$ and is usually denoted
by $Cl_{t}(D)$. The group $Cl_{t}(D)$ was introduced in \cite{B}, where it was pointed
out that $Cl_{t}(D)$ is the divisor class group when $D$ is a Krull domain and
$Cl_{t}(D)$ is the ideal class group when $D$ is a Pr\"{u}fer domain. Also, it was shown in
\cite[Corollary 3.8 and Theorem 3.9]{Z} that an integrally closed AGCD domain is a PVMD with torsion $t$-class group
and that a PVMD with torsion $t$-class group is an AGCD domain.
Thus, an integral domain $D$ is a PVMD with torsion $t$-class group if and only if $D$ is an integrally closed AGCD domain.
For more on the $t$-class group, see \cite{A}.

In Section~\ref{s:1}, we show that if $D[[X]]$ is a $v$-domain, then $D$ is a $v$-domain,
but not necessarily conversely. We also show that if $D$ is an H-domain
(every maximal \mbox{$t$-ideal} of $D$ is divisorial), then $D[[X]]$ is a PVMD if and only if  $D$
is a Krull domain. This answers a question recently raised in \cite{EK}. In Section~
\ref{s:2}, we show that if an integral domain $D$ is completely integrally closed,
then $\bigcap_{n = 1}^{\infty}(I^{n})_{v}= (0)$ for every ideal $I$ of $D$ with $I_v \subsetneq D$
and that the complete integral closure $D''$ of a PVMD $D$ is
a generalized ring of fractions $D_{S}$, where $S$ is the
multiplicatively closed set of $t$-invertible $t$-ideals $I$ of $D$ such that
$\bigcap_{n = 1}^{\infty}(I^{n})_{v} \neq (0)$, thus establishing Theorem \ref{B}.
In this section, we also determine several special classes of PVMDs $D,$ including
PVMDs with torsion $t$-class group, whose being completely integrally
closed requires only that $D$ be Archimedean, i.e., $\bigcap_{n = 1}^{\infty}(x^{n}) = (0)$ for every nonunit $x \in D$.
In Sections~\ref{s:3} and~\ref{s:4}, we continue the work begun in Section
\ref{s:2} and use the notion of a potent maximal $t$-ideal from \cite{ACZ}
to provide new characterizations of integral domains of interest, such as
UFDs, PIDs, Krull domains, and generalized Krull domains. In Section~\ref{s:5},
we investigate several ``Archimedean-like'' conditions for an integral domain.


\section{When $D[[X]]$ is a $v$-domain}       \label{s:1}

We will need the following results from \cite{DH} on ideals in power series rings, in
connection with star operations.

\begin{lemma}  \label{C} \emph{(\cite[Proposition 2.1]{DH})}
Let $I$ be a nonzero fractional ideal of an integral domain $D$.

\emph{(1)} $(ID[[X]])^{-1}$ $=$ $I^{-1}[[X]]$ $=$ $(I[[X]])^{-1}$.

\emph{(2)} $(ID[[X]])_{v}=I_{v}[[X]]=(I[[X]])_{v}$.
\end{lemma}

Lemma \ref{C} was attributed to D. F. Anderson and B.G. Kang in \cite{DH}.
Using Lemma~\ref{C}, we first prove
the following result.

\begin{proposition}  \label{D}
If $D$ is an integral domain such that $D[[X]]$ is a $v$-domain, then $D$ is
a $v$-domain.
\end{proposition}

\begin{proof}
Let $D[[X]]$ be a $v$-domain; so $(JJ^{-1})_{v} = D[[X]]$ for every nonzero finitely generated ideal $J$ of
$D[[X]]$. In particular, let $J = ID[[X]]$ for $I \in f(D)$. Then
$D[[X]] = ((ID[[X]])(ID[[X]])^{-1})_{v} \subseteq  (I[[X]]I^{-1}[[X]])_{v} \subseteq
((II^{-1})[[X]])_{v} = \linebreak
 (II^{-1})_{v}[[X]] \subseteq D[[X]]$ by Lemma~\ref{C}; so $(II^{-1})_v[[X]] = D[[X]]$.
Thus, $(II^{-1})_v = D$ for every $I \in f(D)$; so $D$ is a $v$-domain.
\end{proof}

To see that the converse of Proposition~\ref{D} is not true, note that any $v$-domain $D$ with a nonunit $x$ such that
$\bigcap_{n = 1}^{\infty} (x^{n}) \neq (0)$ can serve as a counterexample. Suppose
that $D[[X]]$ is a $v$-domain. Then, in particular, $D[[X]]$ is integrally
closed. But, by \cite[Theorem 0.1]{Ohm} (or \cite[Theorem 13.10 and Proposition 13.11]{gilmer}), $D[[X]]$ is integrally closed
implies that $D$ is integrally closed and
$\bigcap_{n = 1}^{\infty} (x^{n}) = (0)$ for every nonunit $x \in D$. The presence of a nonunit $x$ with $\bigcap_{n = 1}^{\infty}
(x^{n}) \neq (0)$ will contradict this. Now, take $D$ to be a rank-two valuation
domain and $x \in D$ a nonunit that is not in the height-one prime ideal of $D$; so $\bigcap_{n = 1}^{\infty} (x^{n}) \neq (0)$.
Then $D$ is a PVMD, and hence a $v$-domain, but $D[[X]]$ is not integrally closed, and thus not a $v$-domain.

\begin{remark} \label{D1}
In general, for ideals $I$
and $J$ of $D$, $I[[X]]J[[X]] \subseteq IJ[[X]]$, but $I[[X]]J[[X]] \neq
IJ[[X]]$.  For an example showing that
generally $I[[X]]J[[X]] \neq IJ[[X]]$, see \cite[page 352]{AK}.
\end{remark}

\begin{corollary} \label{E}
If $D$ is an integral domain such that $D[[X]]$ is completely integrally
closed, then $D$ is completely integrally closed.
\end{corollary}

The proof entails noting that $D[[X]]$ is completely integrally closed if
and only if  $(JJ^{-1})_{v} = D[[X]]$  for every non-zero ideal $J$ of $D[[X]]$,
and taking, in particular, $J = ID[[X]]$ for any nonzero
ideal $I$ of $D$ as in the above proof. However, it is well known that $D[[X]]$
is completely integrally closed if and only if $D$ is completely integrally closed.


An integral domain is called a \emph{generalized Krull domain} (cf. \cite[page 524]{gilmer})
if it is a locally finite intersection of essential rank-one valuation domains.
This terminology goes back at least to Griffin \cite{Gr1}, and such rings were considered by Ribenboim \cite{Ribenboim}.
Popescu \cite{NP} introduced the notion of a generalized Dedekind domain via localizing systems.
Nowadays, the following equivalent definition is usually given:
an integral domain is a \emph{generalized Dedekind domain} if it is a
strongly discrete Pr\"{u}fer domain (i.e., $P \neq P^2$ for every prime ideal $P$) and every (prime) ideal $I$ has
$\sqrt{I} = \sqrt{(a_1, \ldots, a_n)}$ for some $a_1, \ldots, a_n \in I$
(or equivalently, every principal ideal has only finitely many minimal prime ideals).
(To add to the confusion, Zafrullah \cite{Z1} defined an integral domain to be a
generalized Dedekind domain if every divisorial ideal is invertible. In \cite{AK1},
these rings were called pseudo-Dedekind domains in analogy with pseudo-principal ideal domains,
i.e., integral domains in which every divisorial ideal is principal.)
Based on the strongly discrete definition of a generalized Dedekind domain,
El Baghdadi \cite{E} defined an integral domain $D$ to be a generalized Krull domain
if it is a strongly discrete PVMD (i.e., $D_M$ is a strongly discrete valuation domain
for every maximal $t$-ideal $M$ of $D$) and every principal ideal has only finitely
many minimal prime ideals, or equivalently, $D$ is a PVMD with $P \neq (P^2)_t$
and $P = \sqrt{J_t}$ for some finitely generated ideal $J$ of $D$ for every prime $t$-ideal $P$ of $D$.
To avoid confusion, we (and hopefully others), will use the terminlogy ``generalized Krull domain''
as defined by Griffin and will call the generalized Krull domains as defined by El Baghdadi \emph{Krull-like PVMDs}.
In particular, Krull-like PVMDs are integrally closed, and generalized Krull domains are completely integrally closed.
While both generalized Krull domains and Krull-like PVMDs are, of course, PVMDs, neither definition implies the other.
For example, while any valuation domain $D$ is a PVMD, $D$ is a generalized Krull domain (resp., Krull-like PVMD)
if and only if $D$ has rank at most one (resp., is strongly discrete).

In \cite[Question 2.4(2)]{EK}, El Baghdadi and
Kim asked the following question: If $D$ is a Krull-like PVMD
domain, is $D[[X]]$ a Krull-like PVMD? The next example gives a negative answer to their question.
The question of \cite{EK} can also be answered in another way via Corollary~\ref{K2}.

\begin{example} \label{F}
Let $D$ be a Krull domain that is not a
field. Then, for every multiplicative subset $S$ of $D$ with at least one
nonunit of $D$, the ring $R = D+YD_{S}[Y]$ is a Krull-like PVMD such
that $R[[X]]$ is not a Krull-like PVMD.
(For a specific example, let $R = \mathbb{Z} + Y\mathbb{Q}[Y]$.)
To see this, note that if $D$ is a Krull domain, then
$D + YD_{S}[Y]$ is a PVMD for every multiplicative subset $S$ of
$D$ \cite[Corollary 2.7]{AAZ2}. Also, as $D$ is a Krull
domain and $D+YD_{S}[Y]$ is a PVMD, $R = D + YD_{S}[Y]$ is a
Krull-like PVMD \cite[Proposition 3.4]{EGZ}. Now, let $d \in D$ be one of
the promised nonzero nonunits in $S$. Then $(0) \neq
YD_{S}[Y]\subseteq \bigcap_{n = 1}^{\infty} d^{n}R$, and as in the discussion concerning the failure of
the ``converse'' of Proposition~\ref{D}, $R[[X]]$ is not integrally closed. Thus, $R[[X]]$ is not a Krull-like PVMD.
\end{example}

We now give the proof of Theorem~\ref{A} from the Introduction.

\begin{proof}
(of Theorem \ref{A}) Let $D = \bigcap_{P \in F}D_{P}$, where
$F$ is a set of divisorial prime ideals of $D$, and suppose that $D[[X]]$ is a
PVMD, and hence a $v$-domain. Every $P \in F$ is divisorial, and so
$P[[X]]$ is a divisorial ideal of $D[[X]]$ by Lemma \ref{C}. It is
well known that if $P$ is a prime ideal of $D$, then $P[[X]]$ is a prime
ideal of $D[[X]]$. Also, every divisorial ideal is a $t$-ideal. So for $P \in F$,
the prime ideal $P[[X]]$ is a $t$-ideal of the
PVMD $D[[X]]$. Thus, $D[[X]]_{P[[X]]}$ is a valuation domain,
and hence $D_P$ is an essential discrete rank-one valuation domain \cite[Theorem 1]{AB}.
Thus, $D_P$ is
completely integrally closed for every $P \in F$,
and hence $D = \bigcap_{P \in F}D_P$ is completely
integrally closed. That $D$ is a $v$-domain follows from Proposition \ref{D}
or the fact that a completely integrally closed domain is a $v$-domain.
\end{proof}

We do not know if the integral domain $D$ in Theorem \ref{A} is actually a PVMD, but for a
$v$-domain $D$ to be a PVMD, all we need to check is that $D$ is a \emph{$v$-finite
conductor domain}, i.e., $aD \cap bD$ is a $v$-ideal of finite type
for every $0 \neq a,b \in D$ \cite[Corollary 4]{FZ}. Thus, we have the following
result.

\begin{corollary}  \label{G}
If $D$ is a $v$-finite conductor domain and $D[[X]]$
is a $v$-domain, then $D$ is a PVMD.
\end{corollary}

In some instances, the $v$-finite conductor property gets provided by
indirect means.

\begin{corollary} \label{H}
Let $D$ be an integral domain that is a locally finite
intersection of localizations at divisorial prime ideals. Then $D[[X]]$ is a PVMD
if and only if $D$ is a Krull domain.
\end{corollary}

\begin{proof}
Suppose that $D[[X]]$ is a PVMD. By the proof of Theorem \ref{A}, $D$ is a locally finite
intersection of discrete rank-one valuation domains. Thus, $D$ is a
Krull domain. Conversely, if $D$ is a Krull domain, then $D[[X]]$
is a Krull domain \cite[Corollary 44.11]{gilmer}, and hence a PVMD.
\end{proof}

Of course, there is yet another way that the $v$-finite conductor condition
becomes available free of charge. Recall that an integral domain $D$ is
called an \emph{H-domain} if for every nonzero ideal $I$ of $D$, $I^{-1} = D$ implies that
there is a finitely generated ideal $F \subseteq I$ such that $F^{-1} = D$. It was
shown by Houston and Zafrullah \cite[Theorem 2.4]{HZ} that $D$ is an H-domain if and only
if every maximal $t$-ideal of $D$ is divisorial. \mbox{H-domains} were introduced
by Glaz and Vasconcelos in \cite{GV}, where it was shown that a completely
integrally closed \mbox{H-domain} is a Krull domain \cite[3.2d]{GV}. Indeed, a Krull domain is an
\mbox{H-domain}. In fact, a Krull-like PVMD is an H-domain since every maximal $t$-ideal
is \mbox{$t$-invertible} \cite[Corollary 3.6]{E}, and hence divisorial.
Thus, a Krull-like PVMD is a Krull domain if and only if it is completely integrally closed.
With this introduction, we state the following result.

\begin{corollary}  \label{K}
The following statements are equivalent for an H-domain $D$.

\emph{(1)} $D[[X]]$ is a PVMD.

\emph{(2)} $D$ is completely integrally closed.

\emph{(3)} $D$ is a Krull domain.

\emph{(4)} $D[[X]]$ is completely integrally closed.
\end{corollary}

\begin{proof}
(1) $\Rightarrow$ (2) Being an H-domain, $D$ is an intersection of
localizations at divisorial prime ideals. By Theorem \ref{A}, $D[[X]]$ is
a PVMD implies that $D$ is completely integrally closed.

(2) $\Rightarrow$ (3) A completely
integrally closed H-domain is a Krull domain \cite[3.2d]{GV}.

(3) $\Rightarrow$ (4) $D$ is a Krull domain implies that $D[[X]]$ is a Krull domain \cite[Corollary 44.11]{gilmer}, and thus $D[[X]]$ is completely
integrally closed.

(4) $\Rightarrow$ (1) $D[[X]]$ is completely integrally
closed implies that $D$ is completely integrally closed, and a completely integrally
closed H-domain is a Krull domain \cite[3.2d]{GV}. Thus, $D[[X]]$ is a Krull domain, and hence a PVMD.
\end{proof}

\begin{corollary}  \label{K2}
Let $D$ be a Krull-like PVMD. Then $D[[X]]$ is a Krull-like PVMD if and only if $D$ is a Krull domain.
\end{corollary}

\begin{proof}
We have already observed that a a Krull-like PVMD is an H-domain. The corollary now follows directly from Corollary~\ref{K}.
\end{proof}

Corollary \ref{K2} shows that the answer to the El Baghdadi-Kim question is, generally no.
Specifically, let $D$ be a Krull-like PVMD that is not a Krull domain (e.g., $\mathbb{Z} + Y\mathbb{Q}[Y]$ as in Example~\ref{F},
or a strongly discrete rank-two valuation domain).
Then $D[[X]]$ is a not a Krull-like PVMD.
(We are thankful to
Said El-Baghdadi for support in the form of advice and references for this
section.)


\section{When $D[[X]]$ is integrally closed for $D$ a PVMD}

\label{s:2}


An integral domain $D$ is \emph{Archimedean} if
$\bigcap _{n = 1}^{\infty}(x^{n}) = (0)$ for every nonunit $x \in D$.
According to \cite[Theorem 0.1]{Ohm} (or \cite[Theorem 13.10 and Proposition 13.11]{gilmer}), if $D[[X]]$ is integrally closed, then
$D$ is integrally closed and Archimedean. Although a completely integrally closed domain is
Archimedean \cite[Corollary 13.4]{gilmer} (cf. Corollary~\ref{L21}),
an Archimedean domain need not be
completely integrally closed since any one-dimensional domain, Noetherian domain,
or more generally, an integral domain satisfying ACCP is Archimedean. However,
we do not know of an example of a PVMD, let alone a Pr\"{u}fer domain, $D$ such that $D[[X]]$ is integrally closed,
but $D$ is not completely integrally closed, or such that $D$ is Archimedean, but not completely integrally closed. So there
seems to be no harm in putting forward a ``lame'' conjecture that if $D$ is a
PVMD such that $D[[X]]$ is integrally closed, then $D$ is completely
integrally closed, which is implied by our second ``lame'' conjecture that a PVMD $D$ is completely
integrally closed if and only if $D$ is Archimedean. These conjectures are ``lame''
in that there is very little hope of them being true. Yet, there is every hope
of generating interest in producing counterexamples to them.

\begin{conjecture}  \label{L}
Let $D$ be a PVMD. Then $D[[X]]$ is integrally closed if and only if $D[[X]]$
is completely integrally closed (if and only if $D$ is completely integrally closed).
\end{conjecture}

\begin{conjecture}  \label{L1}
Let $D$ be a PVMD. Then $D$ is completely integrally closed if and only if $D$ is Archimedean.
\end{conjecture}

These conjectures certainly hold for valuation domains and are somewhat supported by the fact, that will soon become
apparent, that if $D$ is a GCD domain (in fact, an AGCD domain) and $D[[X]]$ is integrally closed, then
$D$ must be completely integrally closed. This follows from the fact that a
GCD domain $D$ is completely integrally closed if and only if $\bigcap_{n = 1}^{\infty}(x^{n}) = (0)$
for every nonunit $x \in D$ (GCD domains satisfy Conjecture \ref{L1}, and hence Conjecture \ref{L}; see Corollary~\ref{P}(1)).

We next define an ``Archimedean-like" condition that \emph{is} equivalent to being completely integrally closed.
We say that an integral domain $D$ is \emph{strongly Archimedean}
if $\bigcap_{n = 1}^{\infty}(a/b)^n = (0)$ for every $a, b \in D$ with $(b) \nsubseteq (a)$.
A strongly Archimedean domain is certainly Archimedean. However, a Noetherian domain is always Archimedean,
but is strongly Archimedean if and only if it is (completely) integrally closed.
Related ``Archimedean-like'' conditions will be studied in Section~\ref{s:5}.

\begin{proposition} \label{L2}
An integral domain $D$ is strongly Archimedean if and only if $D$ is completely integrally closed.
\end{proposition}

\begin{proof}
Let $0 \neq a, b \in D$. Then $\bigcap_{n = 1}^{\infty}(a/b)^n \neq (0)$ if and only $b/a \in D''$,
the complete integral closure of $D$, and $(b) \subseteq (a)$ if and only if $b/a \in D$.
Thus, $D$ is  strongly Archimedean if and only if $D$ is completely integrally closed.
\end{proof}

\begin{corollary} \label{L21}
Let $D$ be a completely integrally closed domain.
Then $\bigcap_{n = 1}^{\infty}(I^n)_v = (0)$ for every ideal $I$ of $D$ with $I_v \subsetneq D$.
In particular, a completely integrally closed domain is Archimedean.
\end{corollary}

\begin{proof}
Since $I_v$ is a proper divisorial ideal of $D$, $I \subseteq (a/b)$ for some $a, b \in D$ with $(b) \nsubseteq (a)$.
Thus, $(I^n)_v \subseteq (a/b)^n$ for every integer $n \geq 1$.
Hence, $\bigcap_{n = 1}^{\infty}(I^n)_v \subseteq \bigcap_{n = 1}^{\infty}(a/b)^n = (0)$
since a completely integrally closed domain is strongly Archimedean.
The ``in particular'' statement is clear.
\end{proof}

The following result sets the stage for a possible resolution of Conjecture
\ref{L1}, at least in some special cases. For better reading, however, we
include some explanation of the notions mentioned in the proposition that
follows.

A nonempty family $S$ of nonzero ideals of an integral domain $D$ is said to be a
\emph{multiplicative system of ideals} if $IJ \in S$ for every $I, J \in S$. If $S$
is a multiplicative system of ideals, then the set of ideals of $D$ each containing some ideal
of $S$ is still a multiplicative system, which is called the \emph{saturation of $S$},
and is denoted by $Sat(S)$. A multiplicative system $S$ is said to be
\emph{saturated} if $S = Sat(S)$. If $S$ is a multiplicative system of ideals, then the
overring $D_{S} = \bigcup\{ \, (D :_K J) \mid J\in S \, \}$ of $D$ is called the \emph{generalized
ring of fractions} (or \emph{generalized transform}) \emph{of $D$ with respect to $S$}. Indeed, $D_{S} = D_{Sat(S)}$, and
note that $\{ \, J_{v} \mid J \in S\,  \} \subseteq Sat(S)$.

\begin{proposition} \label{M}
Let $D$ be a PVMD. Then the complete integral closure $D''$ of $D$
is the generalized ring of fractions $D_S$, where $S = \{ \,I \mid I \subseteq D$ is $t$-invertible
and $\bigcap_{n = 1}^{\infty}(I^{n})_{v} \neq (0) \, \}$ is a multiplicative system of
ideals.
\end{proposition}

\begin{proof}
We first show that $S$ is a multiplicative system of ideals. Let $I, J \in S$, and let
$0 \neq x \in \bigcap_{n = 1}^{\infty}(I^{n})_{v}$ and $0 \neq y \in \bigcap_{n = 1}^{\infty}(J^{n})_{v}$.
Then $IJ$ is $t$-invertible and $0 \neq xy \in (I^n)_v(J^n)_v \subseteq ((I^n)_v(J^n)_v))_v = (I^nJ^n)_v = ((IJ)^n)_v$
for every integer $n \geq 1$.
Thus, $\bigcap_{n = 1}^{\infty}((IJ)^{n})_{v} \neq (0)$; so $IJ \in S$.

We now show that $D'' = D_S$.  Let $x \in D_{S}$. Then $xI \subseteq D$ for some $I \in S$;
so $x^n(I^n)_v = (x^nI^n)_v = ((xI)^n)_v \subseteq D$ for every integer $n \geq 1$.
Let $0 \neq d \in \bigcap_{n = 1}^{\infty}(I^{n})_{v}$.
Then $dx^{n} \in D$ for every integer $n \geq 1$, and so $x \in D''$. Thus,
$D_{S} \subseteq D''$. For the reverse inclusion, suppose that $a/b \in D''$ for $0 \neq a, b \in D$. Then,
there is a $0 \neq d \in D$ such that $d(a/b)^{n} \in D$ for every integer
$n \geq 1$. Hence, $d \in (b^{n}) : (a^{n})$ for every integer $n \geq 1$. Since $D$ is a
PVMD and $(b^{n}):(a^{n})$ is divisorial, we have $
(b^{n}):(a^{n}) = ((b^{n}):(a^{n}))_{w} = \bigcap_{M \in t\text{-Max}
(D)} ((b^{n}):(a^{n}))D_{M} = \bigcap_{M \in t\text{-Max}
(D)} ((b):(a))^{n}D_{M} = (((b):(a))^{n})_w$ for every integer $n \geq 1$ because $D_{M}$ is a valuation
domain for every $M \in$ $t$-Max$(D)$. Thus,
$(b^{n}):(a^{n})=(((b):(a))^{n})_{v}$ for every integer $n \geq 1$ because $(b^{n}):(a^{n})$ is divisorial
and $w \leq v$. Hence, $0 \neq d \in \bigcap_{n = 1}^{\infty}((b^n) : (a^n)) = \bigcap_{n = 1}^{\infty}(((b):(a))^{n})_{v}$.
Thus, $I = (b):(a) \in S$ because $(b):(a)$ is $t$-invertible and
$0 \neq d \in \bigcap_{n = 1}^{\infty}(I^{n})_{v}$.
Hence, $a/b \in D_{S}$ because $(a/b)I = (a/b)((b):(a)) \subseteq
D$; so $D'' \subseteq D_S$. Thus,  $D'' = D_S$.
\end{proof}

We can now give the proof of Theorem~\ref{B} from the Introduction.

\begin{proof}
(of Theorem \ref{B}) Let $D$ be a completely integrally closed PVMD.
Then, by Corollary~\ref{L21}, $\bigcap_{n = 1}^{\infty}(I^{n})_{v} = (0)$
for every proper divisorial ideal $I$ of $D$, and thus $\bigcap_{n = 1}^{\infty}(I^{n})_{v} = (0)$
for every proper $t$-invertible $t$-ideal $I$ of $D$. Alternatively, one can use Proposition~\ref{M}.
Conversely, suppose that $D$ is a PVMD with $\bigcap_{n = 1}^{\infty}(I^{n})_{v} = (0)$
for every proper $t$-invertible $t$-ideal $I$ of $D$.
Then $S = \{D\}$; so $D = D_S = D''$ by Proposition~\ref{M}. Thus,
$D$ is completely integrally closed.
\end{proof}

Recall that an integral domain $D$ is a \emph{generalized GCD} (\emph{GGCD}) \emph{domain} if every finite type $v$-ideal of $D$ is
invertible. GGCD domains were studied in \cite{AA}, where it was shown that
the complete integral closure of a GGCD domain is an invertible
generalized transform. Proposition \ref{M} is an extension of that result. Because
a GGCD domain is a PVMD in which every $t$-invertible $t$-ideal is actually
invertible, \cite[Theorem 5]{AA} and its corollary \cite[Corollary 3]{AA} become special cases of
Proposition~\ref{M} and Theorem~\ref{B}, respectively. Of course, the next corollary is true
for any completely integrally closed integral domain.

\begin{corollary}  \label{N}
If a PVMD $D$ is completely integrally closed, then $D$ is Archimedean.
\end{corollary}

\begin{corollary} \label{P}
\emph{(1)}  \emph{(\cite[Theorem 3.1]{BD})} A GCD domain $D$ is completely integrally closed if and only if $D$
is Archimedean.

\emph{(2)}  \emph{(\cite[Corollary 3]{AA})} A GGCD domain $D$ is completely integrally closed
if and only if $\bigcap_{n = 1}^{\infty}I^{n} = (0)$ for every proper invertible ideal $I$ of $D$.

\emph{(3)}  \emph{(\cite[Corollary 26.9]{gilmer})} A Pr\"{u}fer domain $D$ is completely integrally closed
if and only if $\bigcap_{n = 1}^{\infty}I^{n} = (0)$ for every proper invertible ideal $I$ of $D$.
\end{corollary}

\begin{proof}
Note that a GCD (resp., GGCD or Pr\"{u}fer) domain $D$ is a PVMD in which every $t$-invertible $t$-ideal is principal (resp., invertible).
\end{proof}

We would, of course, like to resolve the two conjectures one way or another.
One way of doing that would be to establish the connection, if one exists,
between a PVMD $D$ being completely integrally closed
and its Kronecker function ring $T$ (or the ring $D\{X\} = D[X]_{N_v}$ \cite{Kan}) being completely integrally
closed. For the Kronecker function ring $T$ (or $D\{X\}$) is a Bezout domain, which being a GCD
domain, is completely integrally closed if and only if $\bigcap_{n = 1}^{\infty}(x^{n}) = (0)$ for every nonunit $x \in T$ (or $D\{X\}$).
In the absence of any insight in that direction,
we are reduced to making the best of the situation.

If we can link every proper $t$-invertible
\mbox{$t$-ideal} $I$ of a PVMD $D$ with a nonunit \mbox{$x \in D$} such that $\bigcap_{n = 1}^{\infty}(I^{n})_{v}\subseteq (x^{m})$
for every integer $m \geq 1$, then $D$ being Archimedean
would be equivalent to $D$ being completely integrally closed. This can be
done in two distinct ways, one computational and the other theoretical; we
pursue both courses.

\begin{lemma} \label{T}
Let $I$ and $J$ be ideals of an integral domain $D$.

\emph{(1)}  If $I \subseteq J$, then $\bigcap_{n = 1}^{\infty}I^{n} \subseteq \bigcap_{n = 1}^{\infty}J^{n}$
and $\bigcap_{n = 1}^{\infty}(I^{n})_{v} \subseteq
\bigcap_{n = 1}^{\infty}(J^{n})_{v}$.

\emph{(2)} If  $I^{k}\subseteq (x)$ for some $x \in D$ and integer $k \geq 1$, then

$\bigcap_{n = 1}^{\infty}I^{n} \subseteq \bigcap_{n = 1}^{\infty}(I^{n})_{v} \subseteq \bigcap_{n = 1}^{\infty}(x^{n})$.
\end{lemma}

\begin{proof}
(1) is obvious. For (2), first note that $\bigcap_{n = 1}^{\infty}(I^{nk})_v = \bigcap_{n = 1}^{\infty}(I^{n})_v$
and $\bigcap_{n = 1}^{\infty}I^{n} \subseteq \bigcap_{n = 1}^{\infty}(I^{n})_{v}$.
If $I^{k}\subseteq (x)$, then $(I^{nk})_v \subseteq (x^n)$ for every integer $n \geq 1$, and thus
$\bigcap_{n = 1}^{\infty}I^{n} \subseteq \bigcap_{n = 1}^{\infty}(I^{n})_v =
\bigcap_{n = 1}^{\infty}(I^{nk})_v \subseteq \bigcap_{n = 1}^{\infty}(x^{n})$.
 \end{proof}

\begin{proposition} \label{Q}
Let $D$ be a PVMD such that for every $0 \neq
a_{1}, \ldots, a_{s} \in D$ with $(a_{1}, \ldots, a_{s})_{v} \neq D$,
there is an integer $k \geq 1$ and a nonunit $d \in D$ such that
$(a_{1}^k, \ldots, a_{s}^k) \subseteq (d)$. Then $D$ is completely integrally closed if and
only if $D$ is Archimedean.
\end{proposition}

\begin{proof}
A completely integrally closed domain is always Archimedean.  For the converse, suppose that $D$ is
Archimedean, that is, $\bigcap_{n = 1}^{\infty}(x^{n}) = (0)$ for every nonunit $x \in D$. Now,
take a proper $t$-invertible $t$-ideal $I$ of $D$. Then $I =
(a_{1}, \ldots, a_{s})_{v}$ for some $0 \neq a_{1}, \ldots, a_{s} \in
D$. By hypothesis, there is an integer $k \geq 1$
and a nonunit $d \in D$ such that $(a_{1}^{k}, \ldots, a_{s}^{k}) \subseteq (d)$,
and thus $(a_{1}^{k}, \ldots,
a_{s}^{k})_{v} \subseteq (d)$. Now, as $D$ is a PVMD, we have
$(a_{1}^{k}, \ldots ,a_{s}^{k})_{v}=((a_{1}, \ldots,
a_{s})^{k})_{v}$ \cite[Lemma 3.3]{AZ}; so $I$ is $t$-invertible and $(I^{k})_{v} \subseteq (d)$.
Hence, $\bigcap_{n = 1}^{\infty}(I^{n})_{v} \subseteq \bigcap_{n = 1}^{\infty}(d^n)$ by Lemma~\ref{T}(2).
But $D$ is Archimedean,
and thus $\bigcap_{n = 1}^{\infty}(I^{n})_{v} \subseteq \bigcap_{n = 1}^{\infty}(d^n) = (0)$
for every $t$-invertible $t$-ideal $I$ of $D$. Hence, $D$ is completely integrally closed by Theorem~\ref{B}.
\end{proof}

As a repeat corollary, we conclude that a GCD domain $D$ is completely
integrally closed if and only if $D$ is Archimedean because in a GCD domain,
$I_{v}$ is principal for every nonzero finitely generated ideal $I$ of $D$. This is, of
course, a known result (\cite[Theorem 3.1]{BD}). Next, an integral domain $D$ with the
\emph{QR property} (every overring of $D$ is a quotient ring) is known to be a Pr\"{u}fer
domain such that for every nonzero finitely generated ideal $I$ of $D$, there is an $i \in I$ and an integer $n \geq 1$ such that
$I^{n} \subseteq (i)$ \cite[Theorem 5]{P}. Thus, we have the following corollary to Proposition \ref{Q}.

\begin{corollary}  \label{R0}
A QR domain $D$ is completely integrally closed if and only if $D$ is Archimedean.
\end{corollary}

We can do somewhat better than Corollary \ref{R0}. Recall that an integral domain $D$ with quotient field $K$
is called a \emph{$t$-QR domain} if every $t$-linked overring of $D$ is a quotient
ring of $D$. Here, in an extension $D \subseteq R \subseteq K$, $R$ is said
to be \emph{$t$-linked over $D$} if
$A^{-1} = D$ implies $(AR)^{-1} = R$ for every nonzero finitely generated ideal $A$ of $D$.
Obviously, every flat overring is $t$-linked; so every overring of a Pr\"{u}fer domain is $t$-linked. In \cite[Theorem 1.3]{DHLZ},
it was shown that a PVMD $D$ has the $t$-QR property if and only if for
every nonzero finitely generated ideal $I$ of $D$, there is a $b \in I_{v}$ and an integer $n \geq 1$ such that
$I^{n} \subseteq (b)$. Since in a Pr\"{u}fer domain every nonzero finitely
generated ideal is a $v$-ideal, this characterization reduces to that given
by Pendleton in \cite[Theorem 5]{P} for QR domains. Thus, we have the
following result as well. (We are thankful to Tiberiu Dumitrescu for
reminding us of \cite{DHLZ}.)

\begin{corollary}   \label{R00}
A $t$-QR PVMD $D$ is completely integrally closed if and only if
$D$ is Archimedean.
\end{corollary}

Proposition \ref{Q} clearly points to the following result
once we note that an integrally closed AGCD domain is a PVMD.
Also, note that Noetherian domains are Archimedean and there are Noetherian AGCD domains that are not integrally
closed.

\begin{corollary}  \label{R}
An integrally closed AGCD domain $D$ is completely integrally
closed if and only if $D$ is Archimedean.
\end{corollary}

It was shown in \cite{Z} that an integrally closed AGCD domain is a PVMD with
torsion $t$-class group and that a PVMD with torsion $t$-class group is an
AGCD domain. Also, since a Pr\"{u}fer domain is a PVMD, a Pr\"{u}fer domain with
torsion class group is an AGCD domain. Hence, a Pr\"{u}fer domain $D$ with torsion
class group is completely integrally closed if and only if $D$ is
Archimedean.  These facts are mentioned here because the QR property was mentioned in
\cite{Ohm}. Of course, Ohm did not know about AGCD domains, nor about $t$-QR domains,
at the time of writing \cite{Ohm}. The results in this section
greatly expand the scope of his work from mere QR domains to $t$-QR and AGCD
domains. Thus, we state the following result.

\begin{proposition}  \label{S}
The following statements are equivalent for an AGCD domain $D$.

\emph{(1)} $D[[X]]$ is integrally closed.

\emph{(2)} $D$ is completely integrally closed.

\emph{(3)} $D[[X]]$ is completely integrally closed.
\end{proposition}

\begin{proof}
(1) $\Rightarrow$ (2) $D[[X]]$ is integrally closed implies that $D$ is
integrally closed and Archimedean by \cite[Theorem 0.1]{Ohm} (or \cite[Theorem 13.10 and Proposition 13.11]{gilmer}).
Thus, $D$ is completely integrally closed by Corollary~\ref{R}.

(2) $\Rightarrow$ (3) $D$ is completely integrally closed implies that $D[[X]]$ is completely integrally
closed.

(3) $\Rightarrow$ (1) This is obvious since a completely integrally closed domain
is integrally closed.
\end{proof}

\begin{corollary} \label{S1}
Let $D$ be an AGCD domain. Then $D[[X]]$ is (completely) integrally closed if and only if
$D$ is a completely integrally closed PVMD with torsion $t$-class group.
\end{corollary}

Corollaries \ref{R0} - \ref{R} use Proposition~\ref{Q} to give several classes of PVMDs
in which being completely integrally closed is equivalent to being Archimedean.
The next result shows that this equivalence actually holds as long as $D''$, the complete integral closure of $D$,
is a quotient ring of $D$.
Since $D''$  is always a \mbox{$t$-linked} overring of $D$ \cite[Proposition 2.5]{AHZ},
Corollary~\ref{R00} also follows from Proposition~\ref{S2}.

\begin{proposition} \label{S2}
Let $D$ be an integral domain with complete integral closure $D'' = D_S$ for a multiplicative subset $S$ of $D$.
Then $D$ is completely integrally closed if and only if $D$ is Archimedean.
\end{proposition}

\begin{proof}
A completely integrally closed domain is always Archimedean.
Conversely, suppose that $D$ is Archimedean. Let $s \in S$. Then $1/s \in D_S = D''$;
so $\bigcap_{n = 1}^{\infty}(s^n) \neq (0)$.
Thus, $s$ is a unit of $D$ since $D$ is Archimedean; so $D = D_S = D''$ is completely integrally closed.
\end{proof}

We end this section with a slight generalization of Proposition~\ref{M} and Theorem~\ref{B}.

\begin{proposition} \label{X7}
Let $D$ be an essential domain. Then the complete integral closure
$D^{\prime \prime }$ of $D$ is the generalized ring of fractions $D_{S}$,
where $S = \{\, I\mid I \subseteq D$ is $v$-invertible and $\bigcap_{n = 1}^{\infty}(I^{n})_{v} \neq (0) \, \}$
is a multiplicative system of ideals.
\end{proposition}

\begin{proof}
Let $D$ be an essential domain with $F$ the set of  essential prime ideals
defining $D$, and let $\ast$ be the star operation induced on $D$ by $F$,
i.e., $I^{\ast} = \bigcap_{P \in F}ID_P$ for every $I \in F(D)$.

The proof is similar to the proof of Proposition~\ref{M}, but we replace the \mbox{$w$-operation}
by $\ast$ as defined in the above paragraph. Note that $(b) : (a) = a^{-1}((a) \cap (b))$ is  \mbox{$v$-invertible}
for every $ 0 \neq a,b \in D$ since $(a) \cap (b)$  is $v$-invertible
The last statement follows because $(((a) \cap (b))(a,b))D_P = (ab)D_P$ for every $P \in F$ since $D_P$ is a valuation domain,
and thus $(((a) \cap (b))(a,b))^{\ast} = (ab)$ implies $(((a) \cap (b))(a,b))_v = (ab)$ since $\ast \leq v$.

\end{proof}

\begin{corollary} \label{X8}
An essential domain $D$ is completely integrally closed if and only if
$\bigcap_{n = 1}^{\infty}(I^n)_v = (0)$ for every proper $v$-invertible $v$-ideal $I$ of $D$.
\end{corollary}

\begin{proof}
Let $D$ be completely integrally closed. Then $D_S = D'' = D$ by
Proposition \ref{X7}, and thus $\bigcap_{n = 1}^{\infty}(I^{n})_{v} = (0)$ for
every proper $v$-invertible $v$-ideal $I$ of $D$. Alternatively, use Corollary~\ref{L21}

Conversely, suppose that $\bigcap_{n = 1}^{\infty}(I^{n})_{v}=(0)$ for every
proper $v$-invertible $v$-ideal $I$ of $D$.
Then, in particular, we have that $(a):(b)$, where $0 \neq a, b \in D$ with $(b) \nsubseteq (a)$,
is \mbox{$v$-invertible} (see the proof of Proposition~\ref{X7}),
and so $\bigcap_{n = 1}^{\infty}(((a):(b))^{n})_{v} = (0)$.
Since $D$ is an essential domain,
$(((a) :(b))^n )_v = (a^n) : (b^n) = ((a^n) : (b^n))_v$ for every integer $n \geq 1$; so
$\bigcap_{n = 1}^{\infty}((a^{n}):(b^{n}))_{v} = (0)$.
But, $(a^{n}):(b^{n}) = (a^{n}/b^{n}) \cap D$,
and so $(0) = \bigcap_{n = 1}^{\infty}((a^{n}/b^{n}) \cap D) =
(\bigcap_{n = 1}^{\infty}(a^{n}/b^{n})) \cap D$, which
forces $\bigcap_{n = 1}^{\infty}(a/b)^n = (0)$ for every
$a, b \in D$ with $(b) \nsubseteq (a)$.
Thus, $D$ is strongly Archimedean, and hence completely
integrally closed by Proposition~\ref{L2}.
\end{proof}

\begin{corollary} \label{X9}
An essential domain $D$ is completely integrally closed
if and only if $\bigcap_{n=1}^{\infty}(((a) : (b))^n)_v = (0)$ for every $0 \neq a, b \in D$ with $(b) \nsubseteq(a)$.
\end{corollary}

\begin{proof}
This follows from the proof of Corollary~\ref{X8} since in an essential domain,
$(((a) : (b))^n)_v = (a^n) :( b^n)$ for every integer $n \geq 1$.
\end{proof}


\section{Alternative approach}

\label{s:3}

The alternative approach is offered here not just to replicate the results
in the previous section, but actually to expand the scope of the study. We
use this approach, for instance, to provide some new characterizations of Krull
domains and their specializations such as UFDs, PIDs, locally factorial Krull
domains, and generalized Krull domains. We do this by
concentrating on integral domains whose maximal $t$-ideals are potent and mixing
them with complete integral closure.

Call a maximal $t$-ideal $P$ of an integral domain $D$ \emph{potent} if it contains a nonzero finitely
generated ideal that is not contained in any other maximal $t$-ideal.
Next, call a $v$-ideal $I$ of finite type a \emph{rigid ideal} if $I$ is
contained in one and only one maximal $t$-ideal. Thus, $P$ is potent if and only if it contains a rigid ideal.
Let us call a rigid ideal contained in a maximal $t$-ideal $P$ a \emph{$P$-ideal}.
It was shown in \cite[Theorem 1.1]{ACZ}
that if $D$ is of \emph{finite $t$-character}, i.e., every
nonzero nonunit belongs to only finitely many maximal $t$-ideals, then every
maximal $t$-ideal of $D$ is potent. Note that every rigid
ideal in a PVMD is a $t$-invertible $t$-ideal.
Also, recall that an integral domain $D$ is a \emph{ring of Krull type} (cf. \cite[page 537]{gilmer})
if $D$ is a locally finite intersection of essential valuation domains. Hence, a generalized Krull domain is a ring of Krull type.
A ring of Krull type is a PVMD, and thus integrally closed, but need not be completely integrally closed.

\begin{lemma}  \label{U}
Let $P$ be a potent maximal $t$-ideal of a PVMD $D$.

\emph{(1)} The set of all $P$-ideals of $D$ is totally ordered by inclusion.

\emph{(2)} If a $t$-invertible $t$-ideal $I$ of $D$
is contained in $P$, then $I$ is contained in a $P$-ideal of $D$.

\emph{(3)} For all $P$-ideals $I$ and $J$ of $D$, the ideal $(IJ)_{v}$ is a $P$-ideal of $D$.

\emph{(4)}  For every $P$-ideal $J$ of $D$, $\bigcap_{n = 1}^{\infty}(J^{n})_{v}$ is a prime ideal
of $D$ contained in $P$.

\emph{(5)} If $J$ is a $P$-ideal of $D$ and $Q$ is a prime ideal of $D$
contained in $P$ with $J \nsubseteq Q$, then $Q \subseteq
\bigcap_{n = 1}^{\infty}(J^{n})_{v}$.
\end{lemma}

\begin{proof}
(1) Let $I$ and $J$ be $P$-ideals of $D$. Then $I_{w} = \bigcap_{Q \in t\text{-Max}
(D)}ID_{Q} = D \cap ID_{P}$, and likewise, $J_{w} = D \cap JD_{P}$. Also, as
$I$ is a $t$-ideal, $I_{w} = I$
because $w\leq t$. Hence, $I = D \cap ID_{P}$, and similarly, $J = D \cap JD_{P}$.
Then, since $D$ is a PVMD and $P$ is a maximal $t$-ideal, $D_{P}$ is a
valuation domain; so $ID_{P} \subseteq JD_{P}$ or $JD_{P} \subseteq ID_{P}$, say
$ID_{P} \subseteq JD_{P}$. Thus, $I = D \cap ID_{P} \subseteq D \cap JD_{P}=J$.

(2) Because $P$ is potent, there is a $P$-ideal $J$ of $D$ contained in $P$. Then $(I + J)_{v}$
is a $P$-ideal of $D$ containing $I$.

(3) Obvious.

(4) Note that $J$ being a $t$-invertible $t$-ideal,
$JD_{P}$ is principal, and so $J^{n}D_{P}$ is principal. Also,
being a $t$-invertible $t$-ideal,
$(J^{n})_{v}D_{P} = (J^{n}D_{P})_{v} = J^{n}D_{P}$ because $J^{n}D_{P}$ is
principal. Now, as $D_{P}$ is a valuation domain, $\bigcap_{n = 1}^{\infty}J^{n}D_{P}
= \bigcap_{n = 1}^{\infty}(J^{n})_{v}D_{P}$ is a prime ideal of $D_P$, say
$QD_{P}$ for $Q \subseteq P$ a prime ideal of $D$. Thus, $Q = QD_{P} \cap D =
(\bigcap_{n = 1}^{\infty}(J^{n})_{v}D_{P}) \cap D = \bigcap_{n = 1}^{\infty}((J^{n})_{v}D_{P} \cap D)
= \bigcap_{n = 1}^{\infty}(J^{n})_{v}$. For the last equality, we used the fact that
$(J^n)_vD_P \cap D = (J^n)_v$ since $(J^n)_v$ is a $P$-ideal of $D$ and $ID_P \cap D = I$
for every $P$-ideal of $D$ as shown in the proof of (1).

(5) Indeed, if $J\nsubseteq Q$, then
$JD_{P}\nsubseteq QD_{P}$, and so $Q \subseteq QD_P \subseteq JD_{P}$ because $D_{P}$ is a
valuation domain. Thus, $Q \subseteq QD_P \subseteq (J^n)_vD_P$,
and hence $Q \subseteq \bigcap _{n = 1}^{\infty}(J^n)_v$ as in the proof of (4).
\end{proof}

These considerations immediately give the following result since a ring of Krull type is a PVMD.

\begin{proposition}  \label{V}
The following statements are equivalent for a ring $D$ of Krull type.

\emph{(1)} $D$ is completely integrally closed.

\emph{(2)} $\bigcap_{n = 1}^{\infty}(I^{n})_{v} = (0)$ for every proper $t$-invertible $t$-ideal $I$
of $D$.

\emph{(3)} $\bigcap_{n = 1}^{\infty}(J^{n})_{v} = (0)$ for every rigid
ideal $J$ of $D$.

\emph{(4)} Every maximal $t$-ideal of $D$ has height one.

\emph{(5)} $D$ is a locally finite
intersection of rank-one valuation domains.
\end{proposition}

\begin{proof}
(1) $\Rightarrow$ (2) This follows from Theorem \ref{B} since a ring of Krull type is a PVMD, or use Corollary~\ref{L21}.

(2) $\Rightarrow$ (3) Obvious.

(3) $\Rightarrow$ (4) Let $P$ be a maximal $t$-ideal of $D$ and suppose
that there is a nonzero prime ideal $Q$ of $D$ contained in $P$ with $x \in
P \setminus Q$. By \cite[Theorem 1.1]{ACZ}, there exists a $P$-ideal $J$ of $D$. Then $L = (J,x)_{v}$ is a (rigid)
$P$-ideal of $D$ not contained in $Q$. Thus, by Lemma~\ref{U}(5), $Q \subseteq
\bigcap_{n = 1}^{\infty}(L^{n})_{v} = (0)$; so $Q = (0)$, a contradiction. Hence, every
maximal $t$-ideal of $D$ has height one.

(4) $\Rightarrow$ (5) A ring of Krull type is a locally finite intersection of
localizations at maximal $t$-ideals, and the localizations are valuation
domains. By (4), each of these localizations has rank one because every
maximal $t$-ideal has height one.

(5) $\Rightarrow$ (1) Since a rank-one valuation domain is completely integrally closed, $D$ is
an intersection of completely integrally closed domains, and thus is completely
integrally closed.
\end{proof}

The above result provides various characterizations
of generalized Krull domains since a generalized Krull domain is a ring of Krull type, most of them are well known. We now
prove some results that \emph{are} new. Yet, to facilitate the realization of those
results, we need to bring in some other notions and results.

Recall from \cite{AZ2} that a family $F$ of (nonzero)
prime ideals of an integral domain $D$ is called a \emph{defining family of primes for $D$} if
$D = \bigcap_{P \in F}D_P$.  If, further, every nonzero
nonunit of $D$ belongs to at most finitely many members of $F$, then $F$ is of
\emph{finite character}, and if no two members of $F$ contain a common nonzero
prime ideal, then $F$ is \emph{independent}. An integral domain $D$ is \emph{independent of
finite character} $F$ (or an \emph{F-IFC domain}) if it has a defining family $F$ of
prime ideals that is independent and of finite character. In \cite{AZ2}, we denoted
by $\ast_{F}$ the star operation induced on $D$ by the family
$\{D_P\}_{P \in F}$, i.e., $I^{\ast_F} = \bigcap_{P \in F}ID_P$ for every $I \in F(D)$.
We also called an integral ideal $I$ of $D$ \emph{$\ast_{F}$-unidirectional}
if $I$ belongs to a unique member of $F$. If $F$ consists of
all the maximal $t$-ideals of $D$, then $\ast_{F} = w$; and if $D$ is a PVMD,
then the rigid ideals are precisely the $w$-invertible unidirectional $w$-ideals
because a $t$-invertible $t$-ideal is a $w$-invertible $w$-ideal.
If $D$ is a GCD domain, then the rigid ideals are precisely the
ones generated by rigid elements of $D$ \cite[Lemma 1]{Zr2}. (For an integral domain $D$, $r \in D$ is said to be \emph{rigid} if $r | t$ and $s | t$ for $s, t \in D$ implies that either $r | s$ or $s | r$). While the unidirectional
ideals served a somewhat limited purpose in \cite{AZ2}, results proved in
\cite{AZ2} can have some interesting uses. One of the results that we intend
to use is the following.

\begin{theorem}  \label{W}
\emph{(\cite[Theorem 3.3]{AZ2})}
Let $F$ be a defining family of mutually incomparable prime ideals of an integral domain $D$
such that $\ast_{F}$ is of finite character. Then the following statements are
equivalent.

\emph{(1)} $F$ is independent of finite character.

\emph{(2)} Every nonzero
prime ideal of $D$ contains an element $x$ such that $(x)$ is a \mbox{$\ast_{F}$-product} of unidirectional $\ast_{F}$-ideals.

\emph{(3)} Every nonzero prime ideal
of $D$ contains a unidirectional $\ast_{F}$-invertible \mbox{$\ast_{F}$-ideal}.

\emph{(4)} For $P \in F$ and $0 \neq x \in P, xD_{P}\cap D$ is $\ast_{F}$-invertible and unidirectional.

\emph{(5)} $F$ is independent and for every nonzero
ideal $I$ of $D$, $I^{\ast_{F}}$ is of finite type whenever $ID_{P}$ is
finitely generated for every $P \in F$.
\end{theorem}

\begin{theorem} \label{X}
A PVMD $D$ is a generalized Krull domain if and only if
$D$ is completely integrally closed and every maximal $t$-ideal of $D$ is potent.
\end{theorem}

\begin{proof}
Suppose that the PVMD $D$ is completely integrally closed and every maximal $t$-ideal of $D$ is potent.
We first show that every maximal $t$-ideal $P$ of $D$ has height one. Since $P$ is potent, it contains at least one
rigid ideal $J$. Since $D$ is completely integrally closed, $\bigcap_{n = 1}^{\infty}(I^{n})_{v} = (0)$ for
every $t$-invertible $t$-ideal $I$ of $D$ by Theorem~\ref{B}; in
particular, $\bigcap_{n = 1}^{\infty}(J^{n})_{v} = (0)$. As in the proof
of (3) $\Rightarrow $ (4) of Proposition \ref{V}, one can show that $P$ has
height one. Thus, every nonzero prime ideal of $D$ contains a maximal $t$-ideal.
Note that $\ast_F = w$, where $F = t$-Max$(D)$, and hence $\ast_F$ has finite character.
Since every maximal $t$-ideal is potent,
Theorem \ref{W}(3) holds. Thus, $F$ has finite character by Theorem~\ref{W}(1); so $D$ is a generalized Krull domain. The converse is obvious.
\end{proof}

Since a Pr\"{u}fer domain is a special case of a PVMD in which $t$-invertible is
simply invertible, we have the following corollary.

\begin{corollary}  \label{Y}
A Pr\"{u}fer domain $D$ is a generalized Krull domain if and
only if $D$ is completely integrally closed and every maximal ideal of $D$ is potent.
\end{corollary}

A happy fallout from Theorem~\ref{X} that may not be termed as corollaries is the
following set of results. Note that the ``completely integrally closed'' hypothesis is needed in Corollary~\ref{ZA}
(let $D$ be a rank-two valuation domain with principal maximal ideal).

\begin{theorem}
\label{Z}Let $D$ be an integral domain that is not a field.

\emph{(1)} $D$ is a PID if and only if every maximal ideal $M$ of
$D$ is principal and $\bigcap_{n = 1}^{\infty}M^{n} = (0)$.

\emph{(2)} $D$ is a Dedekind domain if
and only if every maximal ideal $M$ of $D$ is invertible and
$\bigcap_{n = 1}^{\infty}M^{n} = (0)$.

\emph{(3)} $D$ is a UFD if and only if every maximal $t$-ideal $M$ of $%
D $ is principal and $\bigcap_{n = 1}^{\infty}M^{n} = (0)$.

\emph{(4)} $D$ is a locally factorial Krull
domain if and only if every maximal $t$-ideal $M$ of $D$ is invertible and
$\bigcap_{n = 1}^{\infty}M^{n}=(0)$.

\emph{(5)} $D$ is a Krull domain if and only if every maximal
$t$-ideal $M$ of $D$ is $t$-invertible and $\bigcap_{n = 1}^{\infty}(M^{n})_{v} = (0)$.
\end{theorem}

\begin{proof}
For each of (1) - (5), the ``$\Rightarrow$'' implication is well known and easy to prove.

(2)  (resp., (1)) ($\Leftarrow$) Here, $D$ is an integral domain
in which every nonzero prime ideal is invertible (resp., principal).
It is then well known (via a Zorn's Lemma argument, see \cite[Exercise 36, page 44]{Kap} (resp., \cite[Exercise 10, page 8]{Kap}))
that every nonzero ideal of $D$ is invertible (resp., principal).

(5) ($\Leftarrow$)  Let $M$ be a maximal $t$-ideal of $D$.
Then, by hypothesis, $M$ is $t$-invertible and $\bigcap_{n=1}^{\infty}(M^n)_v = (0)$.
Thus, $MD_M$ is principal; so $Q = \bigcap_{n=1}^{\infty}(MD_M)^n$ is a prime ideal of $D_M$.
As in the proof of Lemma~\ref{U}(4), $Q \cap D = \bigcap_{n=1}^{\infty}(M^n)_v$.
Hence, $Q \cap D = (0)$; so $Q = (0)$, and thus ht$(M) =$ ht$(MD_M) = 1$.
Hence, every maximal $t$-ideal of $D$ has height one and is $t$-invertible.
Thus, every prime $t$-ideal of $D$ is $t$-invertible;
so $D$ is a Krull domain \cite[Theorem 2.3]{HZ}.

(4) (resp., (3)) ($\Leftarrow$)  Note that invertible (resp., principal)
implies $t$-invertible, and so by (5) ($\Leftarrow$), $D$ is a Krull domain
in which every prime $t$-ideal is invertible (resp., principal),
and hence $D$ is locally factorial (resp., a UFD).
\end{proof}

\begin{corollary}  \label{ZA}
Let $D$ be a completely integrally closed PVMD that is not a
field.

\emph{(1)} $D$ is a PID (resp., Dedekind domain) if and only if every maximal ideal
of $D$ is principal (resp., invertible).

\emph{(2)} $D$ is a UFD (resp., locally factorial Krull domain,
Krull domain) if and only if every maximal $t$-ideal of $D$ is principal
(resp., invertible, $t$-invertible).
\end{corollary}

Now we return to the main theme of this work and prove a result with
reference to power series.

\begin{proposition} \label{ZB}
Let $D$ be a PVMD such that every maximal $t$-ideal of $D$ is potent and some
power of every integral $t$-invertible $t$-ideal is contained in a proper principal integral ideal.
Then the following statements are equivalent.

\emph{(1)}) $D[[X]]$ is integrally
closed.

\emph{(2)} $D$ is Archimedean.

\emph{(3)} $D$ is a generalized Krull domain in which some power of every proper integral
\mbox{$t$-invertible} $t$-ideal is contained in a proper principal integral ideal.

\emph{(4)} $D$ is completely integrally closed.
\end{proposition}

\begin{proof}
(1) $\Rightarrow$ (2) This follows from \cite[Theorem 0.2]{Ohm} (or \cite[Proposition 3.11]{gilmer}).

(2) $\Rightarrow$ (4) This follows from Lemma~\ref{T} and Theorem~\ref{B}.

(4) $\Rightarrow$ (3) This follows from Theorem~\ref{X}.

(3) $\Rightarrow$ (2)  A generalized Krull domain is completely integrally closed, and hence Archimedean.

(4) $\Rightarrow$ (1) $D$ completely integrally closed implies that $D[[X]]$
is completely integrally closed, and thus integrally closed.
\end{proof}

\begin{corollary}  \label{ZC}
Let $D$ be a GCD domain such that every maximal $t$-ideal of $D$ is potent. If
$D[[X]]$ is integrally closed, then $D$ is a locally finite intersection of
(essential) valuation domains.
\end{corollary}

\begin{corollary}  \label{ZD}
\emph{(cf.  \cite[Corollary 1.9]{Ohm})}
Let $D$ be an integral domain that is a finite intersection of valuation
domains. If $D[[X]]$ is integrally closed, then $D$ is a finite
intersection of rank-one valuation domains, and hence is a one-dimensional Bezout domain.
\end{corollary}

Generalized Krull domains do not behave as well as Krull
domains in at least one respect. While $D$ is a Krull domain implies that $D[[X]]$ is a Krull domain,
$D$ is a generalized Krull domain implies that $D[[X]]$ is a generalized Krull domain
only when $D$ is a Krull domain. This result of \cite[Theorem 2.5]{PT}
shows that $D$ is a Krull domain when $D[[X]]$ is a generalized Krull domain.
So it does not matter whether we assume that $D[[X]]$ is a generalized Krull domain
as defined by Griffin or defined by El Baghdadi (which we call a Krull-like PVMD), $D$ has to be a Krull domain.


\section{Unique representation domains}    \label{s:4}

In the absence of a clear answer to the two conjectures for PVMDs in general, we look for
special cases, as we have done above. One special case is when a PVMD is a
unique representation domain. A \emph{packet} of an integral domain $D$ is a
$t$-invertible $t$-ideal of $D$ having prime radical.
Then $D$ is called a \emph{unique representation domain} (\emph{URD}) \cite{EGZ} if every proper
$t$-invertible $t$-ideal of $D$ can be uniquely expressed as a $t$-product of
pairwise \mbox{$t$-comaximal} packets. (In \cite{Zr3}, the term URD was used in the more
restricitive sense as a GCD domain that is also a URD.)
We note that for $D$ a PVMD  and $I$ a proper $t$-invertible $t$-ideal of $D$,
if $I$ is a $t$-product of a finite number of packets,
then $I$ can be uniquely expressed as a $t$-product of a finite number
of pairwise \mbox{$t$-comaximal} packets \cite[Theorem 1.1]{EGZ} and if $I= (I_1 \cdots I_n)_t$
is such an expression, then $I = I_1 \cap \cdots \cap I_n$, and that
$I$ has such a representation precisely when $I$ has only finitely many
miniml prime ideals \cite[Theorem 1.2]{EGZ}. Moreover, a PVMD is a URD if and only
if every nonzero principal ideal has only finitely many minimal prime ideals \cite[Theorem, 1.9]{EGZ}.
It may be hoped that this approach will come
in handy if it may look hard to decide on the potency of maximal $t$-ideals,
but there is this finiteness condition. It is well known that a minimal
prime ideal $P$ of a principal ideal $(x)$ is a prime $t$-ideal, and so $D_{P}$ is a
valuation domain when $D$ is a PVMD.

\begin{lemma}  \label{ZE}
Let $D$ be a PVMD and $I$ a packet of $D$ with unique minimal prime ideal $P$.
Then $\bigcap_{n = 1}^{\infty}(I^n)_v = (0)$ if and only if $P$ has height one.
\end{lemma}

\begin{proof}
($\Leftarrow$) Suppose that ht$(P) = 1$. Let $M \supseteq P$ be a maximal $t$-ideal of $D$.
Thus, $ID_M$ is a principal ideal of the valuation domain $D_M$ and $ID_M \subseteq PD_M$,
where ht$(PD_M) = 1$. Then $(0) = \bigcap_{n = 1}^{\infty}(ID_M)^n  =
\bigcap_{n = 1}^{\infty}(I^nD_M) = \bigcap_{n = 1}^{\infty}(I^n)_vD_M \supseteq \bigcap_{n = 1}^{\infty}(I^n)_v$;
so $\bigcap_{n = 1}^{\infty}(I^n)_v = (0)$.

($\Rightarrow$)
Suppose that ht$(P) > 1$. Let $M \supseteq P$ be a maximal $t$-ideal of $D$.
Now, $(I^n)_vD_M = I^nD_M = (ID_M)^n$ for every integer $n \geq 1$ and $\sqrt{ID_M} = PD_M$. Since $D_M$
is a valuation domain, $Q' = \bigcap_{n = 1}^{\infty}(I^n)_vD_M =
\bigcap_{n = 1}^{\infty}(ID_M)^n$ is a prime ideal of $D_M$;
in fact, $Q'$ is the unique prime ideal directly below $PD_M$.
Thus, $Q' = QD_M$, where $Q$ is the unique prime ideal of $D$ directly below $P$.
So ht$(P) > 1$ implies $Q \neq (0)$. Suppose that $N$ is a maximal
$t$-ideal of $D$ with $N \not \supset P$; so $(I^n)_vD_N = D_N$.
Hence, $QD_N \subseteq \bigcap_{n = 1}^{\infty}(I^n)_vD_N$.
Thus, $Q = \bigcap_{M \in t\text{-Max}
(D)}QD_M =  \bigcap_{M \in t\text{-Max}
(D)}(\bigcap_{n = 1}^{\infty}(I^n)_vD_M)  =
\bigcap_{n = 1}^{\infty}(\bigcap_{M \in t\text{-Max}(D)}(I^n)_vD_M)) =
\bigcap_{n = 1}^{\infty}(I^n)_v = (0)$, a contradiction.
\end{proof}

\begin{corollary}  \label{ZF}
Let $D$ be a PVMD, and let $I$ be a proper $t$-invertible $t$-ideal of $D$ such that
$I$ has only a finite number of minimal prime ideals $P_1, \ldots, P_m$. Then
$\bigcap_{n = 1}^{\infty}(I^n)_v = (0)$ if and only if some $P_i$ has height one.
\end{corollary}

\begin{proof} \label{ZG}
By \cite[Theorem 1.2]{EGZ}, $I = (I_1 \cdots I_m)_t$, where $I_1, \ldots, I_m$ are
$t$-comaximal packets each wih $I_i$ contained in a unique minimal prime ideal $P_i$.
By $t$-comaximality, we have $(I^n)_v = (I_1^n \cdots I_m^n) = (I_1^n)_v \cap \cdots \cap (I_m^n)_v$.
Hence, $\bigcap_{n = 1}^{\infty}(I^n)_v =
(\bigcap_{n = 1}^{\infty}(I_1^n)_v) \cap \cdots \cap (\bigcap_{n = 1}^{\infty}(I_m^n)_v)$.
Thus, $\bigcap_{n = 1}^{\infty}(I^n)_v = (0)$ if and only if some
$\bigcap_{n = 1}^{\infty}(I_i^n)_v = (0)$, if and only if ht$(P_i) = 1$.
\end{proof}

\begin{proposition}
Let $D$ be a PVMD URD. Then $D$ is a generalized Krull domain
if and only if $D$ is completely integrally closed.
\end{proposition}

\begin{proof}
A generalized Krull domain is completely integrally closed.
Conversely, let $D$ be a PVMD URD that is completely integrally closed.
Let $x$ be a nonzero nonunit of $D$. Then $(x) = (I_{1} \cdots I_{m})_t$,
where $I_1, \ldots, I_m$ are pairwise $t$-comaximal $t$-ideals with $P_i$
the unique minimal prime ideal containing $I_i$. Then $P_1, \ldots, P_m$ are the
minimal prime ideals of $(x)$. Now $D$ is completely integrally closed;
so $\bigcap_{n = 1}^{\infty}(I_i^n)_v = (0)$ for every $1 \leq i \leq m$ by Corollary~\ref{L21}.
Hence, by Lemma~\ref{ZE}, every $P_i$ has height one. So every prime ideal minimal
over $(x)$ has height one and there are only finitely many of them.
To show that $D$ is a generalized Krull domain, it is enough to show that
$D$ has no maximal $t$-ideal of height greater than one.
By way of contradiction, assume that $M$ is a maximal $t$-ideal of $D$
with a nonzero prime ideal $Q \subsetneq M$. Let $x \in M \setminus Q$.
We can shrink $M$ to a prime ideal $P$ minimal over $(x)$; so $P$ has height one.
Now the prime ideals contained in $M$ are totally ordered.
Since $x \notin Q$, we must have $Q \subsetneq P$, a contradiction.
\end{proof}

Because a ring of Krull type is a PVMD URD \cite[Corollary 1.10]{EGZ}, we have the following corollary.

\begin{corollary} \emph{(cf. Proposition~\ref{V})}
A ring $D$ of Krull type is a generalized Krull domain if and
only if $D$ is completely integrally closed.
\end{corollary}

An integral domain $D$ is a \emph{generalized UFD} (\emph{GUFD}) if $D$ is a GCD domain that satisfies
(1) every nonzero nonunit of $D$ is expressible as a finite product of rigid elements of $D$ and
(2) every rigid element $r$ of $D$ is such that for every factor $s$ of $r$, $r | s^n$ for some integer $n \geq 1$.
Equivalently, $D$ is a generalized Krull domain with $Cl_t(D) = 0$.
The interested reader may consult \cite{AAZ1} for various other characterizations of these integral domains.

These results lead to the following result.

\begin{theorem}  \label{ZK}
The following statements are equivalent for a GCD domain $D$ that is also a URD.

\emph{(1)} $D$ is an intersection of rank-one
valuation domains.

\emph{(2)} $D$ is completely integrally closed.

\emph{(3)} $D[[X]]$ is integrally closed.

\emph{(4)} $D$ is Archimedean.

\emph{(5)} Every packet of $D$ is contained in a height-one prime ideal of $D$.

\emph{(6)} $D$ is a GUFD.
\end{theorem}

\begin{proof}
(1) $\Rightarrow$ (2) $\Rightarrow$ (3) $\Rightarrow$ (4) were
established in \cite[Theorem 0.2]{Ohm}.

(4) $\Rightarrow$ (5) follows from Lemma~\ref{ZE}.

(5) $\Rightarrow$ (6) In a GCD URD, every nonzero nonunit $x$ is expressible as a finite product of
mutually co-prime packets, i.e., elements with unique minimal prime ideals. Say
$x = x_{1} \cdots x_{n}$, where every $(x_{i})$ has a unique minimal prime $P_{i}$,
which by (5) has height one. But then $D_{P_{i}}$ is a rank-one
valuation domain, making $x_{i}$ a rigid element such that for every nonunit
factor $r$ of $x_{i}$, $x_{i}$ divides $r^{n}$ for some integer $n \geq 1$. This makes every
packet a prime quantum and $D$ a GUFD, as described in \cite[page 402]{AAZ1} and in
Zafrullah's doctoral dissertation \cite{Zd}.

(6) $\Rightarrow$ (1) As shown
in \cite[Theorem 10]{AAZ1}, a GUFD is a generalized Krull domain, and thus is
a locally finite intersection of rank-one (essential) valuation domains.
\end{proof}

Ohm proved the equivalence of (1) through (5) of Theorem~\ref{ZK} for finite intersections of
valuation domains \cite[Corollary 1.9]{Ohm}. As a GCD domain of finite $t$-character is a URD as well, we have
the following repeat corollary.

\begin{corollary}  \label{ZL}
A GCD domain $D$ of finite $t$-character is a GUFD if and only if
$D[[X]]$ is integrally closed, equivalently, if and only if $D$ is completely integrally closed.
\end{corollary}

Of course, it would be interesting to see if $D$ is a PVMD URD and $D[[X]]$
is integrally closed implies that $D$ is completely integrally
closed. As it stands, we can only make decisions about GCD domains and
AGCD domains, even for the URD case. We have kept the AGCD URD case as the
last item because it is different from the GCD case in only a few minor
details.

Call an integral domain $D$ an \emph{almost GUFD} if $D$ is a generalized Krull
domain with torsion $t$-class group. Of course, being a generalized Krull
domain, every nonzero nonunit $x \in D$ is expressible as a $t$-product
$(x) = (I_1 \cdots I_n)_{t}$, where $I_i = xD_{P_i}
\cap D$ and $P_i$ ranges over all the height-one prime ideals of $D$ containing $x$
\cite[Corollary 2.3]{AMZ}. Now, as $Cl_t(D)$ is torsion, there are
integers $n_{i} \geq 1$ such that $(I_i^{n_{i}})_{t}$ is a principal
$P_{i}$-primary ideal of $D$.

\begin{theorem} \label{ZM}
The following statements are equivalent for an AGCD domain $D$ that is also a URD.

\emph{(1)} $D$ is a locally finite intersection of rank-one
valuation domains.

\emph{(2)} $D$ is an intersection of rank-one
valuation domains.

\emph{(3)} $D$ is completely integrally closed.

\emph{(4)} $D[[X]]$ is integrally closed.

\emph{(5)} $D$ is integrally closed and Archimedean.

\emph{(6)} $D$ is integrally closed and every packet of
$D$ is contained in a height-one prime ideal of $D$.

\emph{(7)} $D$ is an almost GUFD.
\end{theorem}

\begin{proof}
(1) $\Rightarrow$ (2) is clear.

(2) $\Rightarrow$ (3) $\Rightarrow$ (4) $\Rightarrow$ (5) were
established in \cite[Theorem 0.2]{Ohm}.

(5) $\Rightarrow$ (3) follows from Corollary~\ref{R}.

(3) $\Rightarrow$ (6) follows from Corollary~\ref{L21} and Lemma~\ref{ZE}.

(6) $\Rightarrow$ (7) In an AGCD URD, every nonzero nonunit $x$ is
expressible as a finite $t$-product of mutually $t$-comaximal packets.
Say $(x) = (I_{1} \cdots I_{n})_{t}$,
where every $I_{i}$ has a unique minimal prime ideal $P_{i}$, which by (6) has
height one. Thus, $D_{P_i}$ is a rank-one valuation domain, making
$I_{i} = xD_{P_i} \cap D$. Hence, by \cite[Corollary 2.3]{AMZ}, $D$ is a
generalized Krull domain that is also an AGCD domain.

(7) $\Rightarrow$ (1) This follows since a generalized Krull domain is a locally
finite intersection of rank-one (essential) valuation domains.
\end{proof}


\section{Archimedean-like conditions}     \label{s:5}

In this final section, we consider several ``Archimedean-like'' conditions on an integral domain $D$.

\begin{theorem}  \label{LT}
Consider the following statements for an integral domain $D$.

(1) $D$ is completely integrally closed.

(2) $\bigcap_{n = 1}^{\infty}(a/b)^n = (0)$ for every $a, b \in D$ with $(b) \nsubseteq (a)$, i.e., $D$ is strongly Archimedean.

(3) $\bigcap_{n = 1}^{\infty}((a^n) : (b^n)) = (0)$  for every $a, b \in D$ with $(b) \nsubseteq (a)$.

(4)  $\bigcap_{n = 1}^{\infty}(I^n)_v = (0)$ for every proper $v$-ideal $I$ of $D$.

(5)  $\bigcap_{n = 1}^{\infty}I^n = (0)$ for every proper $v$-ideal $I$ of $D$.

(6)  $\bigcap_{n = 1}^{\infty}(((a) : (b))^n)_v = (0)$  for every $0 \neq a, b \in D$ with $(b) \nsubseteq (a)$.

(7)  $\bigcap_{n = 1}^{\infty}((a) : (b))^n = (0)$  for every $a, b \in D$ with $(b) \nsubseteq (a)$.

(8)  $\bigcap_{n = 1}^{\infty}(x^n) = (0)$ for every nonunit $x \in D$, i.e., $D$ is Archimedean.

Then we have the following implications.
\[
\xymatrix{
(1) \ar@2{<->}[r] & (2) \ar@2{<->}[r] & (3) \ar@2{->}[r] & (4)
\ar@2{->}[r] \ar@2{<->}[d] & (5) \ar@2{<->}[d] & \\
& & & (6) \ar@2{->}[r] & (7) \ar@2{->}[r] & (8)
}
\]
\end{theorem}

\begin{proof}
(1) $\Leftrightarrow$ (2)  follows from Proposition 2.3.

(2) $\Leftrightarrow$ (3) follows from the fact that
$(a^{n}):(b^{n}) = (a/b)^{n} \cap D$ for every $a,b \in D$ with $(b) \nsubseteq (a)$ and integer $n \geq 1$.

(1) $\Rightarrow$ (4) follows from Corollary~\ref{L21}.

(4) $\Rightarrow$ (5) and (6) $\Rightarrow$ (7) are both clear.

(4) $\Leftrightarrow$ (6) and (5) $\Leftrightarrow$ (7) Note that if $I$ is a proper divisorial ideal of $D$,
then $I \subseteq (a/b) \cap D = (a) : (b)$ for some $a, b \in D$ with $(b) \nsubseteq (a)$.

(7) $\Rightarrow$ (8)  is clear.
\end{proof}

The following examples show that none of the ``$\Rightarrow$" implications in the above theorem can be reversed.

\begin{example}   \label{EL}
(a)  ((4) $\nRightarrow$ (3)) Let $D = k[[X^2,X^3]]$ for a field $k$ (or let $D$ be any one-dimensional
Noetherian Gorenstein domain that is not Dedekind). Then $D$ is not completely integrally closed and every proper
nonzero ideal of $D$ is divisorial. Since $D$ is Noetherian, $\bigcap_{n=1}^{\infty}I^{n} = (0)$ for
every proper nonzero ideal $I$ of $D$, and
as $D$ is one-dimensional Gorenstein, every $I^{n}$ is divisorial. Thus,
$\bigcap_{n=1}^{\infty}(I^{n})_{v} = (0)$ for every proper $v$-ideal of $D$.
Hence, (4) $\nRightarrow$ (1); equivalently, (4) $\nRightarrow$ (3).

(b)   ((5) $\nRightarrow$ (4)) \cite[Example 1.5]{HKM} gives a Noetherian domain $D$
with a maximal $t$-ideal $P$ such that $(P^n)_v = P$ for every integer $n \geq 1$.
Then $\bigcap_{n = 1 }^{\infty}P^n = (0)$ since $D$ is Noetherian,
but $\bigcap_{n = 1}^{\infty}(P^n)_v = P \neq (0)$. Thus, (5) $\nRightarrow$ (4); so also (7) $\nRightarrow$ (6).
 (We are thankful to Evan Houston for this example.)

(c)  ((8) $\nRightarrow$ (7)) Let $V = K + M$ be a non-Noetherian (i.e., non-discrete)
one-dimensional valuation domain with maximal ideal $M$ and $K$ a field that is a subring of $V$.
Then $M^2 = M$; so $M^n = M$ for every integer $n \geq 1$. Suppose that $K$ has a proper subfield $k$.
Then $D = k + M$ is also one-dimensional (but not a valuation domain), and thus satisfies (8).
Let $0 \neq m \in M$, $\alpha \in K \setminus k$, $b  = m$, and $a = \alpha m$.
Then $(b) \nsubseteq (a)$ and $(a) : (b) = M$;
so $\bigcap_{n = 1}^{\infty}((a) : (b))^n = \bigcap_{n = 1}^{\infty}M^n = M \neq (0)$.
Thus, (8) $\nRightarrow$ (7).
\end{example}

For an essential domain, we have (3) $\Leftrightarrow$ (6) (see the proof of Corollary~\ref{X8}),
and thus statements (1) - (4), (6) are all equivalent. In a GCD domain,
or more generally an integrally closed AGCD domain, statements (1) - (8) are all equivalent by Corollary~\ref{R}.
In (4) and (5), we may replace ``$I$ is a proper $v$-ideal of $D$'' with ``$I$ is an ideal of $D$
with $I_v \subsetneq D$. Conjecture \ref{L1} is that statements (1) - (8) are all equivalent for a PVMD.

\bigskip
ACKNOWLEDGMENT
\newline
 We would like to thank the referee for a careful reading of the paper and several helpful suggestions.

\end{document}